\documentclass[11pt,oneside]{amsart}
\usepackage[english]{babel}
\usepackage{fullpage}
\usepackage{amssymb,pstricks,amscd,epsfig}
\usepackage{graphicx}

\usepackage{psfrag}
\usepackage{setspace}
\usepackage{paralist}

\usepackage{amsmath, amsthm, amssymb}
\usepackage{pdfsync}
\usepackage[latin1]{inputenc}
\usepackage{stmaryrd}
\usepackage{color}
\usepackage{setspace}
\usepackage[matrix,arrow,tips,curve]{xy}

\numberwithin{equation}{section}

\makeatletter
 \renewcommand\section{\@startsection {section}{1}{\z@}%
     {-4.5ex \@plus -1ex \@minus -.2ex}%
     {2.3ex \@plus.8ex}%
    {\centering\scshape}}

\def\C{\mathcal{C}}
\def\E{\mathcal{E}}
\def\F{\mathcal{F}}

\def\I{\mathcal{I}}

\def\V{\mathcal{V}}
\def\U{\mathcal{U}}

\def\P{\mathcal{P}}

\def\PP{\mathbb{P}}
\def\CC{\mathbb{C}}
\def\ZZ{\mathbb{Z}}

\def\Pic{\operatorname{Pic}}
\def\Hom{\operatorname{Hom}}
\def\Ext{\operatorname{Ext}}

\def\supp{\operatorname{Supp}}

\def\End{\operatorname{End}}

\def\cliff{\operatorname{Cliff}}

\def\wt{\widetilde}
\def\wh{\widehat}
\def\oc2{\mathcal{O}_{\C_2}}

\newtheorem{theorem}{Theorem}[section]
\newtheorem{prop}[theorem]{Proposition}
\newtheorem{lem}[theorem]{Lemma}
\newtheorem{defin}[theorem]{Definition}
\newtheorem{rem}[theorem]{Remark}

\newtheorem{cor}[theorem]{Corollary}

\newtheorem{innercustomthm}{Theorem}
\newenvironment{customthm}[1]
  {\renewcommand\theinnercustomthm{#1}\innercustomthm}
  {\endinnercustomthm}

\newcommand{\be}{\begin{equation}}
\newcommand{\ee}{\end{equation}}

\begin{document}

\title[Mukai program]{
Mukai's program for curves on a K3 surface}

\author{E. Arbarello,   A. Bruno,  E. Sernesi}

\email{ea@mat.uniroma1.it, bruno@mat.uniroma3.it, sernesi@mat.uniroma3.it}
\address{Dipartimento di Matematica, Sapienza, Universit\`a di Roma, Roma. }
\address{ Dipartimento di Matematica e Fisica, Universit\`a Roma Tre, Roma.}
\address{ Dipartimento di Matematica e Fisica, Universit\`a Roma Tre, Roma.}

\maketitle

\begin{abstract}Let $C$ be a general element in the locus of curves in $M_g$ lying on some K3 surface, where
$g$ is congruent to 3 mod 4 and greater than or equal to 15. Following Mukai's ideas, we show how to reconstruct the K3 surface as a Fourier-Mukai transform of a Brill-Noether locus of rank two vector bundles on $C$.
\end{abstract}

\tableofcontents

\section{Introduction}\label{intro}

Let $\mathcal{K}_g$ be the moduli stack of pairs $(S,H)$ where $S$ is a K3 surface and  $H$ is a very ample line bundle on $S$ such that
$H^2=2g-2$. Let $\mathcal{P}_g$ be the stack of pairs $(S,C)$ such that $(S,H)\in \mathcal{K}_g$ and $C\in |H|$  is smooth and irreducible. Finally let $\mathcal{M}_g$ be the moduli stack of smooth curves of genus $g$. The stacks $\mathcal{K}_g, \mathcal{P}_g, \mathcal{M}_g$  are smooth Deligne-Mumford stacks of dimensions $19, 19+g, 3g-3$ respectively.  We have natural morphisms:
\begin{equation}\label{stacks}
\xymatrix{
& \mathcal{P}_g \ar[dr]^-{\kappa_g} \ar[dl]_-{m_g} \\
\mathcal{M}_g&&\mathcal{K}_g }
\end{equation}
where $\kappa_g$ realizes $\mathcal{P}_g$ as an open subset of a $\PP^g$-fibration.
In \cite{Cil-Lop-Mir}, Theorem 5,  the authors prove that for $g=11$ and $g\geq13$ the morphism  $m_g$ is a birational map onto its image using properties of the Gauss map for the canonical divisor (also known  as Wahl map). Related references are \cite{Wahl-square}, \cite{Wahl-sections}.

In the last page of  \cite{Mukai-BN-Fano}, Mukai laid out a beautiful program to actually reconstruct a K3 surface from a curve lying on it,
thus giving a rational inverse of $m_g$,
whenever the genus $g$ is congruent to $3$, mod $4$ and greater or equal than $11$,  and this program
was successfully carried out by him, for the case  $g=11$, in \cite{Mukai2}.

 In our work we take Mukai's 
paper \cite{Mukai2} as a blueprint and generalize it to all genera which are congruent to $3$, mod $4$ and greater or equal than $11$.

Let $(S, C)$ be a general point in $\P_g$, with $g=2s+1, \ s\ge 5$ odd. Mukai's strategy  to reconstruct the surface $S$ from the curve $C$ is as follows:
 consider the Brill-Noether locus $M_C(2, K_C, s)$ which is the moduli space of semistable rank-two vector bundles on $C$ having canonical determinant and possessing at least $s+2$ linearly independent sections. Then  $M_C(2, K_C, s)$
is a K3 surface and the surface $S$ can be obtained as an appropriate Fourier-Mukai transform of it.

When $g=11$ the proof consists of three main steps. One first  considers pairs $(S',C')$ where $S'$ is a  K3 surface of a special type, and proves with ad-hoc constructions that $M_{C'}(2,K_{C'},s)$ is isomorphic   to $S'$ by showing that both are isomorphic to a moduli space $M_v(S')$ of vector bundles on $S'$. 
The second step consists in deforming $(S',C')$ to a general pair $(S,C)$:  since $M_C(2,K_C,5)$  has expected dimension equal to two it is a flat deformation of $M_{C'}(2,K_{C'},s)$, thus it is again   a K3 surface.   Finally  one shows the existence of an appropriate polarization $h$ on $M_C(2,K_C,5)$ which induces an isomorphism between $S$ and the Fourier-Mukai transform  of $M_C(2,K_C,5)$ with respect to $h$.   

The first difficulty in trying to extend this proof is that when $g=2s+1$, $s\geq 6$,  the expected dimension of $M_C(2, K_C, s)$ is zero for $s=6$ and negative for $s\ge 7$, so that it is not even clear that $M_C(2, K_C, s)$ is non-empty when $s \ge 7$. 
 However, in her  paper \cite{Voisin_W},  Voisin associates a  rank-two vector bundle $E_L$   to each base-point-free pencil $|L|$ on $C$ of degree $s+2$. Each of these bundles 
is exhibited as an extension 
$$
0\to K_CL^{-1}\to E_L\to L\to 0
$$
and one can  prove  that Voisin's bundles $E_L$ are stable, ( see, for instance, Lemma \ref{stab_pic_z}, Proposition \ref{s_to_c} and Remark \ref{stab_voisin2}) and that, as $L$ varies in $W^1_{s+2}(C)$, they describe a one-dimensional locus in $M_C(2, K_C, s)$.

 Consider on the K3 surface $S$  the Mukai vector $v=(2, [C], s)$
and denote by $M_v(S)$ the moduli space of $[C]$-stable, rank-two vector bundles $E$ on $S$ with $c_1(E)=[C]$ and $\chi(S,E)=s+2$.
For a general K3 surface $M_v(S)$ is again a smooth K3 surface. One of our main results is the following:

\begin{customthm}{\ref{Tmain}} For a general   $(C,S)\in \P_g$, $g=2s+1$, $s\geq 5$, there is a unique irreducible component  $V_C(2, K_C, s)$ of $M_C(2, K_C, s)$ containing the Voisin's bundles $E_L$.
By sending $\E$ to $\E_{|C}$ one obtains a   well defined isomorphism
\be \label{iso-sigma}
\aligned
\sigma: M_v(S)\to V_C(2, K_C, s)_{red}
\endaligned
\ee
 In particular
$V_C(2, K_C, s)_{red}$ is a smooth K3 surface.
\end{customthm}

Note that we only assumed $g$ to be odd in Theorem \ref{Tmain}. Let now $M_C(2, K_C)$ be the moduli space of rank two vector bundles on $C$ with determinant 
equal to $K_C$. 
Write, for simplicity, $T=V_C(2, K_C, s)_{red}\subset M_C(2, K_C)$. Following Mukai' s program, 
 let 
 $\U$ be   a universal bundle on  $C\times T$, 
 let $\pi_C$ and $\pi_T$ be the natural projections from $C\times T$ to $C$ and $T$, respectively,
 and consider the {\it determinant of the cohomology}
 $$
 h_{det}=(\det R^1{\pi_T}_*\,\U)\otimes(\det{\pi_T}_*\,\U)^{-1}
 $$
For $s$ odd, (i.e. $g\equiv 3\mod 4$), we  prove that $ h_{det}$ is a genus $g$ polarization on $T$ and that $\U$ can be chosen in such a way that the map
$$
\aligned
&C\longrightarrow \wh T=M_{\wh v}(T)\\
&x\mapsto \U_{|\{x\}\times T}
\endaligned
$$
is an embedding and we have the following theorem (see also the more detailed statement in \S \ref{FM}).

\begin{customthm}{\ref{torelli_type}}Let $(C,S)$ be  a general point of $\P_g$, where $g=2s+1$, and $s$ is odd and greater than or equal to $5$.
Let $T=V_C(2, K_C, s)_{red}$.
Consider the Mukai vector $\wh v=(2, h_{det}, s)$. Then any K3 surface containing $C$ is isomorphic to 
$\wh T=M_{\wh v}(T)$.
\end{customthm}

In proving both Theorem \ref{Tmain} and Theorem  \ref {torelli_type}, the basic tool consists  in degenerating the surface $S$ to   a rather special
K3 surface where both the geometry of the moduli space $M_v(S)$ and  the properties of the morphism $\sigma: M_v(S)\to V_C(2, K_C, s)_{red}$ are made transparent by virtue of  an explicit isomorphism $S\cong M_v(S)$ (see Proposition \ref{S-to-M}).

The special K3 surfaces we consider are the direct generalizations of those considered by Mukai in his analysis of the genus 11 case. Namely, we consider a K3 surface $S$ such that 
\be\label{pic-2}
\Pic(S)=\ZZ\cdot [A]\oplus\ZZ\cdot [B]\,,
\ee
with 
\be\label{a-piu-b}
|C|=|A+B|\,,\qquad \text{and }\quad g(A)=s\,,\quad g(B)=1
\ee
This means that the elliptic  pencil $|B|$ cuts out on $C$ a $g^1_{s+1}$ which we call $\xi$, while the linear system $|A|$ cuts out on $C$ the residual series, a $g^s_{3s-1}$ which we call $\eta$.
An isomorphism 
\be\label{iso-rho}
\rho:\, S\longrightarrow M_v(S)
\ee
 is obtained by assigning to each $x\in S$ the vector bundle $\E_x$ defined as the unique extension
$$
0\to\mathcal{O}_S(B)\to\E_x\to I_x(A)\to 0
$$
The isomorphism $\rho$ makes $S$ {\it self-dual}, from the Fourier-Mukai point of view. Such self-duality  is the key to prove  Theorem \ref{Tmain} for pairs $(C,S)$ satisfying (\ref{pic-2}) and (\ref{a-piu-b}).  Moreover, in this case  $S\cong M_C(2, K_C,s)=V_C(2, K_C,s)$  (Theorem \ref{caso(s+1)gon}). The geometry of the special surface $S$ is quite different from the $g=11$ case, and requires a number of new auxiliary results that are proved in \S \ref{notation}. Moreover the negativity of the expected dimension of $M_C(2, K_C,s)$ is the reason for some lengthening in the proof of Theorem 
\ref{caso(s+1)gon}.

The embedding of $S$ in $\PP^s$ via the linear series $|A|$ also
plays a fundamental role. 
Denote by $T$ and $\Gamma$ the images of $S$ and $C$ respectively,  via this embedding. Then, the quadratic hull of $T$ coincides with the quadratic hull of $\Gamma$ and, as such, it classifies extensions on $C$:
$$
0\to\xi\to E\to \eta\to 0
$$
The rank-two vector bundles $E$,  or better their stable models, obtained in this way, parametrize $M_C(2, K_C,s)$
giving, via (\ref{iso-rho}),  a geometrical interpretation of the isomorphism (\ref{iso-sigma}).

In the general case Theorem \ref{Tmain} is proved by a variational argument similar to Mukai's, with the added difficulty coming from the negativity of the expected dimension of the Brill-Noether locus. We consider a family of pairs $(S,C)$ with a special fibre satisfying (\ref{pic-2}) and containing a general pair with Picard rank one among its fibres. By applying some deformation theory arguments we are able to control the behaviour of the map (\ref{iso-rho}) on the general fibre, overcoming the fact that  $M_C(2, K_C,s)$ has negative expected dimension.

\vskip 1 cm 
{\it Acknowledgments: }We heartily thank Claire Voisin 
for suggesting a correction and for a number of very useful remarks.
We would also like to thank  Vittoria Bussi,
Claudio Fontanari, and Giulia Sacc\`a  for  interesting conversations on the subject of this paper.
We are grateful to the referee for making a number of very useful remarks that have   contributed to improve the paper
significantly.


\section{Moduli of sheaves on a K3 surface}\label{sheaves_on_s}

Let $(S,C)$ be a pair consisting of a K3 surface $S$ and a nonsingular curve $C \subset S$ of genus
$$
g(C)=g=2s+1, 
$$
for some $ s \ge 5$.  We let $M_{v,C}(S)$ be the moduli space   of $[C]$-semistable sheaves
with Mukai vector $v$ on $S$ and polarization $[C]$. The Mukai vector of a sheaf $F$ is given by
$$
v(F)= (r(F), c_1(F),\chi(F)-r(F))
$$
Where $r(F)$ denotes the rank of $F$. 
From now on  we consider the case in which the Mukai vector $v$ is given by
\be\label{mukai_v}
v= (2, [C],s)
\ee
 and we will write
$$
M_{v,C}(S)=M_{v}(S)
$$
As already anticipated in the Introduction, in this paper we will mostly consider the following two cases, to which we give a name.

 \begin{itemize}
	 \item  \emph{Rank-1} case: $\Pic(S)=\ZZ\cdot[ C]$.

\item \emph{Rank-2} case:  $\Pic(S)=\ZZ\cdot [A]\oplus\ZZ\cdot [B]$, with $[C]=[A+B]$, $A$ and $B$ nonsingular connected and 
$g(A)=s$, $g(B)=1$.   We then write
$$
\mathcal{O}(B)_{|C}=\xi\,,\qquad \mathcal{O}(A)_{|C}=\eta
$$
so that $\xi$ is a $g^1_{s+1}$ and $\eta$, the residual of $\xi$, is a $g^{s}_{3s-1}$. For the existence of K3 surfaces of this type
we refer to Theorem 1.1 in \cite{Knutsen-smooth}.
\end{itemize}

\begin{lem}\label{nowalls}
 In both the Rank-1 and the Rank-2 cases the Mukai vector  (\ref{mukai_v}) is primitive and there are no walls relative to $v$.
\end{lem}

\begin{proof}
In the  Rank-1 case the lemma is obvious. Let us assume that we are in the Rank-2 case. The curve $C$ is clearly primitive and therefore $v$ is primitive as well. 
Let us  show that there are no walls relative to $v$ in the ample cone of $S$ (see \cite{Huybrechts-Lehn} Definition 4.C.1). 
Following the notation of Theorem 4.C.3 in  \cite{Huybrechts-Lehn}, it suffices to prove that there is no element 
 $$
\lambda=2c_1(F')-c_1(F)
 $$
 where $F$ is a $\mu_C$-semistable sheaf in $M_{v,C}(S)$
 and $F'\subset F$ is a rank-1 subsheaf with $\mu_C(F')=\mu_C(F)$,
 such that
\be\label{dis_delta}
 -\Delta\leq\lambda^2\leq 0
 \ee
 where $\Delta=4c_2(F)-c_1^2(F)=4(s+2)-4s=8$.
 We may write $c_1(F')=hA+kB$. Now (\ref{dis_delta}) reads
 $$
-8\leq(2h-1)^2(2s-2)+2(2k-1)(2h-1)(s+1)\leq 0
 $$
The equality  $\mu_C(F')=\mu_C(F)$, gives  $2h(s-1)+k(s+1)=2s$ and the above inequalities   can be written as 
 $$
-4\leq-s(2h-1)^2\leq0
 $$
and this  has no solutions for $s\geq5$. So there are no walls.  
\end{proof}

\begin{lem}\label{L:nowalls1}
Assume that the Mukai vector (\ref{mukai_v}) is primitive and there are no walls relative to $v$. Then the moduli space  $M_{v}(S)$, if not empty, is a smooth K3 surface
and all of its points represent {\it locally free} sheaves. In particular this happens if we are in the Rank-1 or in the Rank-2 case.
\end{lem}

\begin{proof}
     Applying Theorem 4.C.3 and Lemma 1.2.13 in \cite{Huybrechts-Lehn} we deduce that all sheaves in $M_{v}(S)$ are $[C]$-stable.  Since $v$ is isotropic it follows that $M_{v}(S)$ is a smooth and irreducible projective surface (\cite{Huybrechts-Lehn}, Theorem 6.1.8) which is a K3 by \cite{Mukai}. Since in this case $[C]$-stability is equivalent to $\mu$-stability from \cite{Huybrechts-Lehn}  Remark 6.1.9 p. 145 it follows that all the points of $M_{v}(S)$ represent locally free sheaves.  The last assertion is a consequence of Lemma \ref{nowalls}.
\end{proof}

From the lemma it follows that, under its assumptions,    $[C]$-(semi)stability is computed in terms of the $C$-slope which is defined by
$$
\mu_C(F)=\frac{c_1(F)\cdot C}{r(F)}
$$
In particular this happens if we are in the the Rank-1 or in the Rank-2 case.
 Let us recall the definition of {\it Lazarsfeld-Mukai bundle}. 

\begin{defin}\label{Laz-Muk_def}  Let  $L$ be a globally generated  pencil on   $C\subset S$. The \emph{Lazarsfeld-Mukai bundle} $\wt E_{L}$ is the dual of the rank-2 vector bundle $\wt F_{L}$  defined by the exact sequence
$$
0\to \wt F_{L}\to H^0(L)\otimes\mathcal{O}_S\overset{ev}\to L\to 0
$$
\end{defin}

\begin{rem}\label{notationLM}\rm Often,  in the literature,  the bundle $\wt F_{L}$ is denoted by the symbol $F_{C,L}$
and its dual bundle $\wt E_{L}$ is denoted by the symbol $E_{C,L}$
\end{rem}

For these bundles one easily computes the basic invariants:

\be\label{inv1}
\aligned
r(\wt E_L)&=2\,,\quad c_1(\wt E_L)=[C]\,,\quad c_2(\wt E_L)=\deg L\,,\\
h^0(\wt F_L)&=h^0(\wt E_L(-C))= h^2(\wt E_L)=0\,,\\
h^1(\wt F_L)&=h^1(\wt E_L(-C))= h^1(\wt E_L) =0\,,\\
h^0(\wt E_L)&=h^0(L)+h^1(L) \endaligned
\ee
As far as the $C$-slope is concerned, we have

\be\label{slope1}
\mu_C(\wt E_L)=2s
\ee

 We will need the following Lemma (see also Remark \ref{stab_voisin2}).
 
\begin{lem}\label{stab_pic_z} Assume that we are in the Rank-1 case. Let $|L|$ be a $g^1_{s+2}$ on $C$.
Then $\wt E_L$ is stable. In particular $M_v(S) \ne \emptyset$.
\end{lem} 

\begin{proof} Observe that   the $g^1_{s+2}$
is automatically base-point-free because, by Lazarsfeld's proof of Petri conjecture, the curve $C$ is Petri
and thus has no $g^1_{s+1}$. Since $\Pic(S)=\ZZ\cdot [C]$ and $c_1(\wt E_L)=[C]$, we may assume that a destabilizing
sequence has the form
$$
0\to\mathcal{O}(nC)\to \wt E_L\to  I_X((1-n)C)\to 0
$$
where $n\geq 1$ and
 $X\subset S$ is a zero-dimensional subscheme. But this is absurd since $h^0(S, \mathcal{O}(nC))>h^0(S, \wt E_L)=s+2$. The last assertion is obvious.
\end{proof}

\begin{defin}Let $L$ be a globally generated pencil on  $C$. The restriction to $C$ of the Lazarsfeld-Mukai bundle $\wt E_L$ is called the Voisin bundle of $L$ and it is denoted by the symbol $E_L$.
\end{defin}

\begin{rem} \rm
It is a very remarkable fact that the bundle $E_L$ only depends on $C$, $L$ and 
the first infinitesimal 
neighborhood $\C_2$ of $C$
in $S$.  In \cite{Voisin_W} Voisin proves that, even more generally, one may construct a bundle having all the properties of $E_L$
starting from $C$, $L$ and  an embedded ribbon 
$$
\C_2\subset \PP^g
$$ 
having $C$ as hyperplane section. She also observes that the datum of an embedded ribbon $\C_2\subset \PP^g$ having $C$ as hyperplane section, is equivalent to the datum of an element $u$
 in the kernel of the dual of the Gaussian map
$$
H^1(C, T_C^2)\to H^1(C, T_{\PP^{g-1}}\otimes T_C)
$$
Moreover she proves that if $R_L\in H^0(C, K_C+2L)$ is the ramification divisor of the map determined by $|L|$, the
class $uR_L\in H^1(C, T_C+2L)$  determines an extension
\be\label{voisin_seq}
0\to L\to E_L\to KL^{-1}\to 0
\ee
which splits at the level of cohomology so that
\be\label{inv2}
h^0(E_L)=h^0(L)+h^1(L)= h^1(E_L)
\ee
\vskip 0.3 cm
  From \ref{inv2} it follows that if  $L$  a general element of $W^1_{s+2}$ then $E_L$ is a rank-two vector bundle on $C$ with determinant equal to $K_C$ and with $s+2$ linearly independent sections. From the Brill-Noether point of view this is most unusual. Certainly  a general curve of genus $2s+1$ admits no such a vector bundle. The next section is devoted to the analysis  of those Brill-Noether loci that are relevant in our study of curves lying on K3 surfaces.
	\end{rem}


\section{Brill-Noether loci for moduli of vector bundles on $C$}\label{bn_vb}
\vskip 0.3cm

Let $(S,C)$ be as in the previous section.  We denote 
by 
$M_C(2, K_C)$ the moduli space of rank two, semistable vector bundles on $C$ with determinant equal to $K_C$. We consider the Brill-Noether locus
$$
M_C(2, K_C, s)= \{[F]\in M_C(2, K_C)\,\,\,|\,\,\, h^0(C, F)\geq s+2\}
$$

A point  $[F]\in M_C(2, K_C)$   corresponding to a stable bundle is smooth for  $M_C(2, K_C)$ and
$$
T_{[F]}(M_C(2, K_C))=H^0(S^2F)^\vee\cong\CC^{3g-3}\,,
$$
It is well known that the Zariski tangent space to the Brill-Noether locus $M_C(2, K_C, s)$ at a point $[F]$  can be expressed in terms
of the ``Petri'' map
\be\label{petri-map} \mu: S^2H^0(F)\longrightarrow H^0(S^2F)
\ee
Indeed
\be\label{tang-BN}
T_{[F]}(M_C(2, K_C, s))=\operatorname{Im}(\mu)^\perp
\ee
In particular $M_C(2, K_C, s)$ has expected dimension $3g-3-{s+3 \choose 2}$.

Also notice that, if $F$ is a rank two vector bundle on $C$ with determinant equal to $K_C$, since  $\chi(S^2F)=\chi(F\otimes F^\vee)-\chi(K_C)=3g-3$, we get:
\be\label{dim-stab}
F\quad\text{stable}\quad\Rightarrow\quad h^0(S^2F)=3g-3\quad\Rightarrow\quad h^1(S^2F)=0\,.
\ee

\begin{prop}\label{s_to_c}  Assume that  we are in the Rank-1 case.
Let $v$ be the Mukai vector (\ref{mukai_v}).
Then there is a well defined morphism
\be\label{mod_s_to_c}
\aligned
\sigma: M_v(S)&\longrightarrow M_C(2, K_C,s)\\
&[\E]\mapsto [\E_{|C}]
\endaligned
\ee

 \end{prop}

\begin{proof}  We first show that for every $[\E]\in M_v(S)$ the vector bundle $\E_{|C}$ is stable.
Suppose this is not the case.
Then there is an exact sequence
\be\label{non_st}
0\to \alpha\to \E_{|C}\to K\alpha^{-1}\to 0
\ee
with  $d=\deg\alpha \geq g-1=2s$.  From  (\ref{non_st}) we get a  diagram 

\be\label{diag_v}
\xymatrix{
&0\ar[d]&0\ar[d] \\
0\ar[r]&\E(-C)\ar[d]\ar@{=}[r]&\E(-C)\ar[d]\\
0\ar[r]&V\ar[d]\ar[r]&\E\ar[d]\ar[r]&K\alpha^{-1}\ar@{=}[d]\ar[r]&0\\
0\ar[r]&\alpha\ar[d]\ar[r]&\E_{|C}\ar[d]\ar[r]&K\alpha^{-1}\ar[r]&0\\
&0&0
}
\ee
The rank two locally free sheaf  $V$ satisfies
$$
c_1(V)=0\,,\qquad c_2(V)=c_2(\E)-c_1(\E)\cdot[C]+c_1(K\alpha^{-1})=s+2-d\leq 2-s<0
$$
In particular $V\cong V^\vee$, but  $V$ is not the trivial bundle. 
We have  
$$
\chi(V)=\chi(\E(-C))+\chi(\alpha)=s+2+d-g+1\geq s+2.
$$ 
Thus $2h^0(V)\geq s+2$ so that $h^0(V)\geq 4$. Therefore  $V$ cannot be stable, otherwise
$\dim\End(V)=h^0(V\otimes V^\vee)=h^0(V\otimes V)=1$.
Suppose  $V$ is strictly semistable,  then by Jordan-H\"older we have an exact sequence
$$
0\to \mathcal{L}\to V\to \mathcal{L}^{-1}\otimes I_X\to 0
$$
where  $X\subset S$ is a $0$-dimensional subscheme and $c_1(\mathcal{L})\cdot C=0$.
We also have
$$
\chi(V)=\chi(\mathcal{L})+ \chi(\mathcal{L}^{-1}\otimes I_X)
$$ 
so that
$$
0<-c_2(V)= c_1(\mathcal{L})^2 - h^0(\mathcal{O}_X) \le c_1(\mathcal{L})^2
$$
 Since   $c_1(\mathcal{L})\cdot C=0$ this contradicts the   fact that  $c_1(\mathcal{L})=hC$ for some integer $h$.
It follows that $V$ contains a destabilizing line subbundle $\mathcal{L}$. 
Therefore   $\mathcal{L}$ satisfies
\begin{equation}\label{destab1}
c_1(\mathcal{L})\cdot C > 0,  \quad c_1(\mathcal{L})^2 \ge -c_2(V) \ge s-2
\end{equation}
Then we must have $c_1(\mathcal{L})=hC$ for some positive integer $h$ and this contradicts the stability of $\E$.
We may then conclude that sending $\E \mapsto \E_{|C}$ gives a well defined    morphism $\sigma':M_v(S)\longrightarrow M_C(2,K_C)$.

  By Lemma \ref{stab_pic_z}, for every  pencil $L$ of degree $s+2$   we have 
$$
[\wt E_L]\in M_v(S)
$$
Since, by (\ref{inv1}), we have $h^0(S, \wt E_L(-C))=h^1(S,\wt E_L)=h^2(S, \wt E_L)=0$, the same relations must be true for a general point  $[\E]\in M_v(S)$. Looking then at the exact sequence
$$
0\to \E(-C)\to \E\to  \E_{|C}\to 0
$$
we easily see that for a general  point  $[\E]\in M_v(S)$ the point $[ \E_{|C}]$ belongs to $M_C(2, K_C,s)$
i.e. $h^0(C,\E_{|C})\geq s+2$. But $M_v(S)$ is a smooth K3 surface, in particular it is irreducible.
Thus the image of $\sigma'$ must be contained in $M_C(2, K_C,s)$.
\end{proof}

\begin{rem}\label{stab_voisin} \rm
The proof of   Proposition \ref{s_to_c} can be extended  to the Rank-2 case using the Hodge Index Theorem but for the last part one needs the existence of a base point free pencil $L$ of degree $s+2$ and the stability of $\wt E_L$. The existence of $L$ will be proved in Proposition \ref{petri_rank2} and the stability in Remark \ref{stab_voisin2}. In Section 
\ref{bn_vb2} we will take a different approach to  the study of $\sigma: M_v(S)\to M_C(2, K_C,s)$ in the Rank-2 case (Corollary \ref{rhork2}).

The stability of Lazarsfeld-Mukai bundles and Voisin's bundles has been studied by various authors.
In the Rank-1 case  the stability of $\wt E_L$ is almost immediate (Lemma \ref {stab_pic_z}) but it is also a consequence
of a more general result in \cite{Lelli-Chiesa}, Section 4. In the same case the stability of $\wt E_L$, together with Proposition \ref{s_to_c}, gives the stability of its restriction $E_L$. This particular result is also a consequence of Theorem 4.1 in \cite{Aprodu-Farkas-Ortega}.  The stability of both $\wt E_L$ and $ E_L$ cannot be deduced from the above mentioned papers.      
\end{rem}

Finally, we want to study the differential of the morphism $\sigma: M_v(S)\to M_C(2, K_C,s)$, whenever it is well defined.

\begin{prop}\label{L:rk1dsigma}Assume that the morphism
$$
\aligned
\sigma: M_v(S)&\to M_C(2, K_C,s)\\
&\E\mapsto \E_{|C}
\endaligned
$$
is well defined. Look at the composition:
\[
\xymatrix{
 \sigma': M_v(S) \ar[r]^-\sigma & M_C(2,K_C,s) \ar@{^(->}[r]^-j & M_C(2,K_C) }
\]
where $j$ is the inclusion.
Let $\mathcal{E}\in M_v(S)$ and consider the following conditions:
\begin{itemize}
	\item[(i)]    $H^1(S, S^2\mathcal{E})=0$.
	
	\item[(ii)] $H^0(S, S^2\mathcal{E}(-C))=0$.
	
	\item[(iii)]  $S^2H^0(S, \mathcal{E})\longrightarrow H^0(S, S^2\mathcal{E})$ is surjective.
	\end{itemize}
Then:	
\begin{enumerate}
\item If (i) is satisfied then  $d\sigma'$ is injective at $[\mathcal{E}]$.
	
	\item If (i),(ii) and (iii) are satisfied then $M_C(2,K_C,s)$ is nonsingular of dimension 2 at $\mathcal{E}_{|C}$ and $\sigma$ is \'etale at $\mathcal{E}$.
	\end{enumerate}
\end{prop}

\proof (1) -
Write $E=\E_{|C}$. We use the isomorphism 
\be\label{iso-sym}
\E^\vee\otimes\E=\E^\vee\otimes\E^\vee(C)\cong (S^2\E^\vee\oplus\wedge^2\E^\vee)(C)\cong S^2\E^\vee(C)\oplus\mathcal{O}_S
\ee
We have
$$
\aligned
T_{[\E]}(M_v(S))&=H^1(S, \E\otimes \E^\vee)=H^1(S, S^2\E^\vee(C)) \\ 
T_{[E]}(M_C(2,K_C))&=H^0(C,S^2E)^\vee=H^1(C, S^2E^\vee(K_C))
\endaligned
$$

Thus $d\sigma'$ is the restriction homomorphism
$$
d\sigma': H^1(S, S^2\E^\vee(C))\to H^1(C, S^2 E^\vee(K_C))
$$
Hence 
$$
\ker \sigma'=H^1(S,  S^2\E^\vee)=H^1(S,  S^2\E)
$$
(2) - Consider the following commutative diagram of maps:
\[
\xymatrix{
S^2H^0(E) \ar[r]^-a & H^0(S^2E) \\
S^2H^0(\mathcal{E}) \ar[u]^-b \ar[r]_-c & H^0(S^2 \mathcal{E}) \ar[u]_-d
\
}
\]
Condition (ii) implies that $d$ is injective and (iii) implies that (c) is surjective. Therefore 
\[
\mathrm{corank}(a) \le \mathrm{corank}(d) \le h^1(S^2\mathcal{E}(-C))
\]
Now consider the decomposition (\ref{iso-sym}).
Since $\mathcal{E}$ is stable we have $h^0(\mathcal{E}^\vee\otimes \mathcal{E})=1$ and $h^1(\mathcal{E}^\vee\otimes \mathcal{E})=2$ and from the decomposition  we deduce that 
$h^1(S^2\mathcal{E}(-C))=2$ as well, using hypothesis (ii).  Therefore  $\mathrm{corank}(a) \le 2$; since coker$(a)^\perp = T_{[E]}M_C(2,K_C,s)$, we conclude  that 
$\mathrm{dim}[T_{[E]}M_C(2,K_C,s)]\le 2$.   But from (i) and  (1) it follows that $M_C(2,K_C,s)$ has dimension $\ge 2$ at $[E]$, and this proves (2).  \qed


\section{Geometry of $(S,C)$ in the Rank-2 case}\label{notation}

As in the previous two sections we denote by $(S,C)$ a pair consisting  of a  K3 surface $S$ and a smooth curve $C\subset S$ of genus
$$
g(C)=2s+1\,,\qquad s\geq 5
$$
In this section  we    assume that we are in the Rank-2 case and we prove a few technical results.  

\begin{lem}\label{a-b}a) $h^0(S, \mathcal{O}( nA+mB))=0$, whenever $n<0$.

b) $h^i(S, \mathcal{O}( A-B))=h^i(S, \mathcal{O}(B-A))=0$, for all $i$.

c) Every element in $|A|$ is integral and has Clifford index $\ge 2$. In particular $|A|$ is very ample.
\end{lem}
\begin{proof}
 a) is immediate by restricting  to $B$.
 
Let us prove b). We have $(A-B)^2=-4$ and, by point a), we also have $h^2(S,\mathcal{O}(A-B))=0$. Hence
$h^1(S, \mathcal{O}(A-B)=h^0(S, \mathcal{O}(A-B))$.
Suppose  $h^0(S, \mathcal{O}(A-B))\neq0$.  Let  $D$ be an effective divisor linearly equivalent to   $A-B$.
 If  $D$ is connected, then by Lemma 2.2  in \cite{Saint-Donat},  we get $h^1(S, \mathcal{O}(D))=0$ and we are done.
Otherwise
$$
D=D_1+D_2\,,\qquad D_1\cdot D_2=0\,,\qquad h^0(S, \mathcal{O}(D_i))\geq 1\,,\qquad i=1,2
$$
$$
D_1=nA+mB\,,\qquad D_2=hA+kB\,,\qquad n+h=1\,,\quad m+k=-1
$$
Since $D_1$ and $D_2 $ are  effective, by a) 
we must have $n\geq 0$, $h\geq 0$.
Thus either $n=1$ and $ h=0$, or $n=0$ and $ h=1$.
In any event, we would get $D_1\cdot D_2\neq0$ which is absurd.

Let us prove the first part of c).
Write $A=\Gamma+\Delta$, with both $\Gamma$ and $\Delta$  effective.
Then $\Gamma=nA+mB$. Since $A\cdot B\geq\Gamma\cdot B=nA\cdot B$, we get $n\leq 1$.
Since $h^0(S, \mathcal{O}(\Gamma))\neq 0$, by a) we must have $n=0,1$. If $ n=0$ then $mB$ must be a subcurve
of $A$ against point b). If $n=1$, since $\Gamma$ is a sub curve of $A$, we must have $m<0$ which
 again violates b).

Let us now prove the second part of point c).
By the main theorem of \cite{Green-Lazarsfeld-special}, and by its refinement in \cite{Ballico-Fontanari-Tasin},  
all curves   $A'\in |A|$ have the same Clifford index and   Cliff$(A')$ is computed 
by the restriction to $A'$ of an invertible sheaf $L$ on $S$. So, in order to complete the proof of c) we must exclude the existence of $L \in \mathrm{Pic}(S)$
with either of the following properties:

\begin{itemize}
	\item[(i)] $L^2=0$ and $L\cdot A'=2$ (i.e. $A'$ is hyperelliptic).
	
	\item[(ii)] $L^2=0$ and $L\cdot A'=3$ (i.e. $A'$ is trigonal).
	
	\item[(iii)]  $s=6$ and $L^2=2$ and $L\cdot A = 5$ (i.e. $A'$ is isomorphic to a nonsingular plane quintic).
\end{itemize}

Let us consider (ii). We must have $L=nA+mB$ and the two conditions translate into
 $$
 (nA+mB)^2=2n^2(s-1)+2nm(s+1)=0\,,\qquad 2n(s-1)+m(s+1)=3\,,\qquad (s\geq 5)
 $$
which are clearly incompatible. The hyperelliptic case (i) is  similar.

In case (iii) we must have:
\[
(nA+mB)^2=2n(5n+7m)=2\,,\qquad 10n+7m=5
\]
implying the impossible identity $5n(n+1)=2$. 

\end{proof}

\begin{rem}\label{consq-a-b}\rm

A) From point c) of Lemma \ref{a-b} it follows that, via the  linear system $|A|$,  the surface $S$ is embedded in 
$\PP^s$ as a projectively normal surface whose ideal is generated by quadrics.

B) Let $I_S$ and $I_C$ be the ideal sheaves of $S$, respectively  $C$, in $\PP^s$. Recall that  $2A-C\sim A-B$ as divisors on $S$.
From point A)  and point b) of Lemma \ref{a-b} we deduce 
that $H^0(S, \mathcal{O}(2A))\cong  H^0(C, 2\eta)$ and in particular we get 
a surjection

\be\label{proj-norm-eta}
m: S^2H^0(C, \eta)\to H^0(C, 2\eta)\to 0
\ee

and an equality
\be\label{proj-norm-eta1}
H^0(S, I_S(2))=H^0(C, I_C(2))=\ker (m)
\ee

\end{rem}

\begin{lem}\label{geq-3}Let $D\subset S$ be a finite closed subscheme of length $d\ge 1$.
Assume that
\begin{equation}\label{E:ineq1}
h^0(S, \mathcal{I}_D(A))\geq \mathrm{max} \left\{3,s- \frac{d-1}{2} \right\}
\end{equation}
 Then $d=1$.
\end{lem}

\begin{proof}
 We   view $S$ embedded in $\PP^s=\PP H^0(S,\mathcal{O}(A))^\vee$. Consider  a   hyperplane $H$ passing through $D$, i.e defining a non-zero element of $H^0(S, \mathcal{I}_D(A))$. 
We set $A=H\cap S$. We may  view $D$ as a subscheme of the integral curve $A$. As such it defines a rank-one torsion free sheaf on $A$ which we still  denote by $D$.
From   (\ref{E:ineq1})  we get 
\be\label{contr_om(-D)}
 h^0(A, \omega_A(-D)) \geq 2
\ee
Thus, by Riemann-Roch on $A$:
 \be\label{contr_D}
 h^0(A, \mathcal{O}_A(D))=h^0(A, \omega_A(-D))+d-s+1\geq \frac{d+1}{2}
 \ee
 Therefore either $h^0(A, \mathcal{O}_A(D)) =1$ and $d+1\leq 2$, implying that  $d\leq 1$, which is precisely what we aim at, or $ h^0(A, \mathcal{O}_A( D))\geq 2$, which, together
 with (\ref{contr_om(-D)}) tells us that $D$ contributes to the Clifford index of $A$.
 Let us see that this case can not occur.
By (\ref{contr_D}) we get 
 \be\label{cliff_D}
 \aligned
 \cliff{D}&=d-2h^0(A, \mathcal{O}_A(D))+2\\
 &\leq d-2\left(\frac{d+1}{2}\right)+2\leq 1\\
 \endaligned
 \ee
 and this implies that Cliff$(A) \le 1$, contradicting Lemma  \ref{a-b} c). \end{proof}

 \begin{rem}\rm In a sense, the technical lemma we just proved is our substitute for  Mukai's Lemma 7 which is ubiquitous in \cite{Mukai2}.

It suggests the possibility of introducing the notion of \emph{Clifford index of a $0$-dimensional closed subscheme} $D$ on a polarized K3 surface $(S,H)$ by letting
\[
\mathrm{Cliff}(D) = 2g-d - 2 h^0(S,\mathcal{I}_D(H))+2
\]
where $d = \mathrm{length}(D)$ and $g = \frac{1}{2}H^2+1$ is the genus of $H$. One says that $D$ \emph{contributes to the Clifford index of $H$} if 
both $h^0(S,\mathcal{I}_D(H)) \ge 3$ and $h^0(S,\mathcal{I}_D(H)) +d-g+1 \ge3$.  
A straightforward generalization of the proof of Lemma \ref{geq-3} gives that, if all $H' \in |H|$ are integral, then
\[
 \mathrm{Cliff}(H) = \mathrm{min}\{\mathrm{Cliff}(D): D \subset S \ \text{contributes to $\mathrm{Cliff}(H)$}\}
\]

\end{rem}

Our next aim is to prove the following Proposition regarding   linear series of degree $s+2$ on $C$ in the  Rank-2 case.

\begin{prop}\label{petri_rank2} 

 a) Let $|L|$ be a degree-$(s+2)$,  base-point-free pencil on $C\subset S$.
Then $L$ is Petri.

b) $h^0(C, \eta\xi^{-1})=1$ and thus $\xi$ is a smooth isolated point of  $W^1_{s+1}$.

c) There exists on $C$ a base-point-free $g^1_{s+2}$.
\end{prop}

We are going to use the following Lemma due to Green-Lazarsfeld  and Donagi-Morrison.
We take its statement from \cite{Ciliberto-Pareschi}, Lemma 2.1.

\begin{lem}\label{gl-dm}(Green-Lazarsfeld and Donagi-Morrison)
Let $S$ be a K3 surface. Let $|L|$ be a base-point-free pencil on a smooth curve $C$ lying on $S$.
If the  Lazarsfeld-Mukai  bundle $\wt E_L$ is not simple (i.e. it has non-trivial automorphisms), then there exists line bundles  $M$ and  $N$ on $S$
and a zero-dimensional subscheme  $X\subset S$ such that

i) $h^0(S,M)\geq 2$, $h^0(S, N)\geq 2$.

ii) $N$ is base-point-free.

iii) There is an exact sequence
\be\label{succ_mn}
0\to M\to \wt E_L\to N\otimes I_X\to 0
\ee
Moreover if  $h^0(S, M-N)=0$ then  $\supp(X)=\emptyset$  and the above sequence splits.
\end{lem}

\begin{proof} {\it (of Proposition \ref{petri_rank2})} As far as point a) is concerned we proceed exactly as in
Lazarsfeld's proof of Petri's conjecture. It is then enough to prove that 
$\wt E_L$ is simple. Suppose it is not so. By the preceding Lemma there is an exact sequence (\ref{succ_mn}).
Since $\Pic(S)=\ZZ\cdot [A]\oplus\ZZ\cdot[B]$, $C=A+B$ and  $c_1(\wt E_{L})=C$ we must have

$$
M=mA+nB\,,\qquad N=hA+kB\,,\qquad (m+h-1)A+(n+k-1)B\sim 0
$$
from condition  i) of the Lemma and from Lemma \ref{a-b} we get that the pair  $\{M, N\}$ coincides with the pair $\{A, B\}$.
On the other hand, by Lemma \ref{a-b},  we have 
$H^0(A-B)=0$. Using Lemma \ref{gl-dm} again, the sequence \ref{succ_mn} splits and   $\wt E_L=\mathcal{O}(A)\oplus\mathcal{O}(B)$.
This is absurd since $h^0(S, \wt E_L)=s+2$.
\vskip 0.2 cm
Regarding item b),  looking at the Petri map $H^0(C,\xi)\otimes H^0(C, \eta)\to H^0(C, \omega_C)$, 
and using the b.p.f.p.t
we see that
$h^0(C,\eta\xi^{-1})\geq 1$. On the other hand, looking
at the exact sequence
$$
0\to\mathcal{O}_S(-2B)\to\mathcal{O}_S(A-B)\to\eta\xi^{-1}\to 0
$$
we see that $h^0(C,\eta\xi^{-1})\leq 1$.
\vskip 0.2 cm

As far as point c) is concerned, consider the smooth locus
$$
V=\{\xi(p)\,|\,p\in C\}\subset W^1_{s+2}
$$
To analyze $W^1_{s+2}$ along $V$ 
we look at
 the Petri map.
$$
\mu_{0, \xi(p)}: H^0(C,\xi(p))\otimes H^0(C,\eta(-p))\to H^0(C,  K_C)
$$
By hypothesis, 
$H^0(C,\xi(p))=H^0(C,\xi)$ and 
$$
\ker\mu_{0, \xi(p)}=H^0(C, \eta\xi^{-1}(-p))
$$
By b), we know that $h^0(C, \eta\xi^{-1})=1$. Let $D\neq0$ be the divisor of a non-zero section of  $\eta\xi^{-1}$.
By Brill-Noether theory it  follows that $V$ is a one-dimensional component of $W^1_{s+2}$ and that
$\xi(p)$ is a singular point of  $W^1_{s+2}$ if and only if
$p\in \supp D$.  Moreover, in this case
$$
\dim T_{\xi(p)}W^1_{s+2}=2
$$
A priori it could be that, for $p\in \supp D$, the dimension of $W^1_{s+2}$ at $\xi(p)$ is equal to $1$ and $\xi(p)$
is an embedded point. However,  for a determinantal variety of the correct dimension this can not be the case.
Thus there must be 
 a one-dimensional component $V'$of $W^1_{s+2}$, distinct from $V$, and  meeting $V$ at $\xi(p)$.
Now suppose that there is no base-point-free $g^1_{s+2}$ in $V'$. Then
$$
V'=\{\xi'(p')\,|\,p'\in C\}
$$
for some fixed
 $\xi'\in W^1_{s+1}$, with $ \xi'\neq\xi$. Therefore there is a point $p'\in C$, such that
$$
\xi(p)=\xi'(p')
$$
 But then   $\xi=\xi'$, contrary to the assumption.

\end{proof}

Proposition \ref{petri_rank2}, together with the results in \cite{Arbarello-Cornalba}, gives the following proposition.

\begin{prop}Set $g=2s+1$. 
Let $M_{g,s+1}^1\subset M_g$ be the irreducible divisor of curves possessing a $g^1_{s+1}$.
Le $[C]$ be a general point in $M_{g,s+1}^1$. Then there is a unique $g^1_{s+1}$ on $C$ and every  degree-$(s+2)$ base-point-free pencil on $C$ satisfies Petri's condition.
\end{prop}


\section{Brill-Noether loci in the Rank-2 case}\label{bn_vb2}


   In this section we assume that we are in the Rank-2 case.  
	Using Lemma \ref{a-b} b) and 
proceeding  as in Section 3 of \cite{Mukai2},
to each point $x\in S$ we associate a rank-2, pure sheaf $\E_x$ defined as the unique extension

\be\label{schwarz} 
0\to \mathcal{O}_S(B)\to\E_x\to I_x(A)\to 0.
\ee
We define 
$$
E_x={\E_x}_{|C}
$$
 The main theorem we want to prove in this section is the following generalization of Theorem 3 of \cite{Mukai2}.

\begin{theorem}\label{caso(s+1)gon}  By associating to $x \in S$ the bundle $E_x$ on $C$ we obtain an isomorphism between $S$ and $M_C(2, K_C,s)$.
\end{theorem}

The proof of the theorem will be obtained from the following chain of facts to be proved:

\begin{itemize}
	\item  Both $E_x$ and ${\E_x}$ are stable. Moreover $[\E_x]\in M_v(S)$ and 
$[E_x]\in M_C(2,K_C,s)$.

\item 
The   map 
\be
\aligned
\rho: S& \longrightarrow M_v(S)\\
&x\mapsto [\E_x]
\endaligned
\ee
is an isomorphism of K3 surfaces.

\item From these two facts it follows that  the morphism
\be\label{mod_s_to_c3}
\aligned
\sigma: M_v(S)&\longrightarrow M_C(2, K_C,s)\\
&[\E]\mapsto [\E_{|C}]
\endaligned
\ee
is well defined.

\item $\sigma$ is bijective.

\item For every $x\in S$ 
$$
\dim T_{[E_x]}(M_C(2,K_C,s))=2
$$
and the differential of $\sigma$ is an isomorphism at $[\E_x]$, for every $x\in S$.

The last two items  finally give: 
 
\item  $\sigma$   is in fact an isomorphism of smooth K3 surfaces and therefore $\sigma\rho: S \longrightarrow M_C(2,K_C,s)$ is an isomorphism.

\end{itemize}

We need to establish a number of preliminary results. First of all, since $H^1(S,\mathcal{O}(B))=0$, from 
(\ref{schwarz}) we get the exact sequence

\be\label{schwarz-0}
0\to H^0(S, \mathcal{O}_S(B))\to H^0(S, \E_x)\to H^0(S, I_x(A))\to 0.
\ee
\be\label{dim-ex}
H^0(S, \E_x)=s+2\,,\qquad H^i(S, \E_x)=0\,,\quad i=1,2
\ee
and $\E_x$ is generated by global sections. Notice that since $\det\E_x=\mathcal{O}(C)$ we have
\be\label{dual-e-x}
\E_x^\vee\cong\E_x(-C)
\ee
From the sequence and (\ref{dim-ex}) we get

\be\label{restr-e-x}
0\to \E_x^\vee\to \E_x\to E_x\to 0
\ee
we then  get an isomorphism

\be\label{e-x-e-x}
H^0(S, \E_x)\cong H^0(C, E_x)
\ee
and in particular 

\be\label{dim-e-x1}
\dim H^0(C, E_x)=s+2=h^0(C,\xi)+h^0(C,\eta)-1\,. 
\ee

If $x\notin C$
restricting (\ref{schwarz}) to $C$ gives
\be\label{restr-schwarz1}
0\to \xi\to E_x\to \eta\to 0\, 
\ee

If $x\in C$ then, factoring out the torsion from $I_x(A)\otimes \mathcal{O}_C$, we get 
\be\label{restr-schwarz-x}
0\to \xi(x)\to E_x\to \eta(-x)\to 0\, 
\ee

We are now going to prove two results  that are key elements in   the proof of the stability of $E_x$
(Proposition \ref{stab-ex}).

\begin{lem}\label{non-split-x}The extension (\ref{restr-schwarz-x}) is non-split.
\end{lem}
\proof
To prove this we proceed  as in the proof of Mukai's Proposition 3 in \cite{Mukai2}.

Let $E_x\to \xi(x)$ be a splitting. 
Consider the diagram
$$
\xymatrix{
&0\ar[d]&0\ar[d]\\
&\E_x(-C)\ar@{=}[r]\ar[d]&\E_x(-C)\ar[d]\\
0\ar[r]&\F\ar[r]\ar[d]&\E_x\ar[r]\ar[d]&\xi(x)\ar[r]\ar@{=}[d]&0\\
0\ar[r]&\eta(-x)\ar[r]\ar[d]&E_x\ar[r]\ar[d]&\xi(x)\ar[r]&0\\
&0&0\\
}
$$
Then $ H^0(S, \F)$ is an $s$-dimensional subspace of $H^0(S, \E_x)$ mapping onto $H^0(S, I_xA)$. We have $c_1(\F)=0$ so that $\wedge^2\F=\mathcal{O}_S$. Look at the evaluation map
$H^0(S, \F)\otimes\mathcal{O}_S\to \F\subset\E_x$
and take its second wedge product
$$
\wedge^2(H^0(S, \F)\otimes \mathcal{O}_S)\overset \alpha\longrightarrow\wedge^2\F=\mathcal{O}_S\subset\wedge^2\E_x=\mathcal{O}_S(C)
$$
where  $\alpha=(c_1,\dots, c_n)$, with $c_i\in \CC$, $i=1,\dots,n$.
Restricting $\alpha$ to $C$ we see that it vanishes on $x$ and therefore
vanishes identically. Thus, the image of the evaluation map 
$H^0(S, \F)\otimes\mathcal{O}_S\to \F\subset\E_x$ is of rank one and is isomorphic to $I_x(A)$, which is a contradiction, since $\E_x$ is non split.

\endproof

 \begin{lem}\label{dest} Let $E$ be a  rank two vector bundle (not necessarily semi-stable) with canonical determinant on $C$. 
 Suppose that
 $ h^0(E)= s+2$ and that $E$ contains a line sub bundle isomorphic to $\xi$ or $\xi(p{)}$ for a point $p\in C$.
 Then either
 
 a) $E$ is stable, or
 
 b) there is an exact sequence $0\to\xi\to E\to \eta\to 0$, and $\eta(-p)$ is a destabilizing subsheaf of $E$.
 
 c) $E=\xi({p})\oplus \eta(-p)$
 \end{lem}
 
 \begin{proof}
 Assume that $E$ is not stable;
 then we must have an exact sequence 
 $$
 0\to\alpha\to E\to \beta\to 0
 $$
 where $\alpha$ is a line bundle on $C$ of degree greater or equal to  $g-1$.
 We have two possible diagrams 
 \be\label{abxieta}
 \xymatrix{
& &0\ar[d]\\
 &&\alpha\ar[d]^a\\
0\ar[r]&\xi\ar[r]^c&E\ar[r]^d\ar[d]^b&\eta\ar[r]&0\,,\\
&&\beta\ar[d]\\
&&0
 }\qquad
 \xymatrix{
& &0\ar[d]\\
 &&\alpha\ar[d]^a\\
0\ar[r]&\xi({p})\ar[r]^c&E\ar[r]^d\ar[d]^b&\eta(-p)\ar[r]&0\\
&&\beta\ar[d]\\
&&0
 }
 \ee
 Since $\deg \alpha\geq2s> s+1=\deg\xi$, (resp. $\deg \alpha\geq2s> s+2=\deg\xi({p})$) there can be no injective map from $\alpha$ to $\xi$ (resp. $\xi({p})$). Thus, 
 we must have $\alpha=\eta(-D)$ and $\beta=\xi(D)$ for some positive divisor of degree $d\leq s-1$ (resp. $d\leq s-2$).
 We have
\be\label{ineq1}
 \aligned
 s+2=h^0(E)&\leq h^0(\xi(D))+h^0(\eta(-D))\\
 &=2h^0(\eta(-D))+s+1+d-g+1\\
 &=2h^0(\eta(-D))-s+1+d
  \endaligned
\ee
 Thus
\be\label{ineq2}
 h^0(\eta(-D))\geq s- \frac{d-1}{2}
\ee

Thus
$$
 h^0(\eta(-D))\geq s- \frac{d-1}{2}\geq\frac{s+2}{2}\geq 3\,,\qquad (\text{resp.}\quad  h^0(\eta(-D))\geq \frac{s+3}{2}\geq 3)
$$
 since $s\geq 5$. We can then apply 
 Lemma \ref{geq-3} and  we get $d\leq 1$. Then the only possibility are  the ones described in points b) and c).
 
 \end{proof} 
 
 The next result needed to prove the stability of $\E_x$ and $E_x$ is Mukai's Lemma 2 in \cite{Mukai2}.
We include its statement for the convenience of the reader.
 
\begin{lem} \label{mukai2} (Mukai) Let $L$ be a line bundle on a smooth curve $C$ and consider non-trivial extensions
$$
0\to L\to E\to M\to 0
$$
with $M=K_CL^{-1}$.

1) The extensions $E$ with $h^0(E)=h^0(L)+h^0(M)$ are parametrized by the projective space
$\PP^*\operatorname{Coker}\{S^2H^0(M)\to H^0(M^2)\}$.

2) Assume that the multiplication map $S^2H^0(M)\to H^0(M^2)$ is surjective. Then $h^0(C, E)\leq h^0(C, L)+h^0(C, M)-1$. 
Moreover, the  non-trivial extensions  $E$ such that 
$h^0(C, E)=h^0(C, L)+h^0(C, M)-1$
are  parametrized by the quadratic hull of the image of $\Phi_{|M|}: C\to \PP^*H^0(C, M)$.
More precisely, for every point $x$ of the quadric hull, there is  a  unique extension $E$ such that the image of the linear map $H^0(C, E)\to H^0(C,M)$ is the codimension one subspace corresponding to $x$.
\end{lem}

We are now ready to prove: 

\begin{prop}\label{stab-ex} Both $E_x$ and ${\E_x}$ are stable. Moreover $[\E_x]\in M_v(S)$ and 
$[E_x]\in M(C,K,s)$.
\end{prop}

\proof   Assume first that $x\notin C$. 
We are then in case (\ref{restr-schwarz1}). By (\ref{dim-e-x1}) the sequence cannot split. By Lemma \ref{dest},
if $E_x$ is not stable then it sits in an exact sequence
$$
0\to \eta(-p)\to E_x\to \xi(p)\to 0
$$
for some $p\in C$ and we have $H^0(C, E_x)\cong H^0(C, \xi(p))\oplus H^0(C,  \eta(-p))=H^0(C, \xi)\oplus H^0(C,  \eta(-p))$. Thus the codimension one subspace of  $H^0(C,  \eta)$ is $H^0(C,  \eta(-p))$.
But then $x=p\in C$ which is a contradiction.

Let then $x\in C$. Then we are in case (\ref{restr-schwarz-x}). By Lemma \ref{non-split-x}, this sequence cannot split.
The stability of  $E_x$ now  follows immediately from Lemma \ref{dest}.

The stability of $\E_x$ is now clear. If $L$ is a destabilizing subsheaf of $\E_x$ then  $L\cdot C\geq C^2/2$.
But then $L_{|C}$ would destabilize $E_x$. The last assertion is a consequence of   (\ref{dim-ex}) and (\ref{dim-e-x1}).
\endproof

\begin{rem}\rm
From the preceding arguments we learned that points $x$ in the quadratic hull of $C\subset \PP^s$, not belonging to $C$, correspond to
 extensions of type  (\ref{restr-schwarz1})  where $E_x$ is a stable bundle.
 On the other hand, points $x$ belonging to $C$ correspond to 
 extensions of type
  $$
a)\qquad  0\to \xi\to E'_x\to \eta\to 0
 $$
where $E'_x$  is  destabilized by $\eta(-x)$.
 Finally, if $D$ is the divisor of a section of $\eta\xi^{-1}$, a point $x\in \supp(D)$, corresponds to an extension of type:
 $$
b)\qquad 0\to \xi\to E''_x=\eta(-x)\oplus\xi(x)\to \eta\to 0
 $$
 and $E''_x$ is clearly unstable. In both cases a) and b) the ''stable limit''
 replacing $E'_x$ (resp. $E''_x$) is $E_x$ as in  (\ref{restr-schwarz-x}) 
 
\end{rem}

We next come to:

\begin{prop}\label{S-to-M} The   map 
\be
\aligned
\rho: S& \longrightarrow M_v(S)\\
&x\mapsto [\E_x]
\endaligned
\ee
is an isomorphism of K3 surfaces.
\end{prop}

\proof   We follow again  Mukai's line of reasoning \cite{Mukai2} (p. 194, before Lemma 8, and p. 195 after Lemma 11). The local to global spectral sequence of ext gives
$$
H^j(S, \E xt^i_{\mathcal{O}_S}(I_x(A),\mathcal{O}_S(B))) \Rightarrow \Ext^{i+j}_{\mathcal{O}_S}(I_x(A),\mathcal{O}_S(B))
$$
By Lemma \ref{a-b} the natural map
\be\label{loc-glob}
\Ext^{1}_{\mathcal{O}_S}(I_x(A),\mathcal{O}_S(B))\longrightarrow H^0(S,  \E xt^1_{\mathcal{O}_S}(I_x(A),\mathcal{O}_S(B)))\cong\CC
\ee
is an isomorphism so that the extension (\ref{schwarz}) is the unique non trivial extension of $I_x(A)$ by $\mathcal{O}_S(B)$.  Now one can perform a relative version of this construction. We let $T$ be a copy of $S$ and $\Delta$ be the diagonal of $S\times T$. We have an isomorphism
$$
\E xt^{1}_{\mathcal{O}_{S\times T}}(I_\Delta(p^*A),\mathcal{O}_{S\times T}(q^*B))\longrightarrow q_*\E xt^1_{\mathcal{O}_{S\times T}}(I_\Delta(p^*A),\mathcal{O}_{S\times T}(q^*B)))\cong\mathcal{O}_T(B-A)
$$
which is a relative version of (\ref{loc-glob}). We then have a universal extension

\be\label{univ-ext}
0\to\mathcal{O}_{S\times S}(p^*B)\to \F\to I_\Delta(p^*A+q^*(B-A))\to 0
\ee
whose restriction to $S\times\{x\}$ is (\ref{schwarz}). This gives a well  defined morphism
$$
\aligned
\rho: S&\to M_v(S)\\
&x\mapsto [\E _x]
\endaligned
$$
As $S$ and $M_v(S)$ are smooth K3 surfaces,  to prove that $\rho$ is an isomorphism  it suffices to show that
it is injective and for this it suffices to show that 
$$
\dim\Hom(O(B), \E_x)=1\,,\qquad \text{i.e}\quad  h^0(S, \E_x(-B))=1
$$
But this follows readily from the exact sequence
$$
0\to \mathcal{O}_S\to \E_x(-B)\to I_x(A-B)\to 0
$$
\endproof

\begin{cor}\label{rhork2}
 $\sigma: M_v(S)\longrightarrow M_C(2,K_C,s)$ is well defined. 
\end{cor}

\proof The corollary is an immediate consequence of Propositions \ref{stab-ex} and \ref{S-to-M}.
\endproof

\begin{prop}\label{sigma-bij}
$\sigma$ is bijective.
\end{prop}
 
\proof  
- $\sigma$ is injective.

Clearly what we have to prove is that $\dim \Hom(\xi,  E_x)=1$, or in other words that $h^0(C, E_x\xi^{-1})=1$.
From  the exact sequence
$$
0\to \mathcal{O}_S(-B-C)\to \E_x(-B)\to E_x\xi^{-1}\to 0
$$
 we get
$$
H^0(S, \E_x(-B))\cong H^0(C, E_x\xi^{-1})
$$
From the sequence
$$
0\to \mathcal{O}_S\to \E_x(-B)\to I_x(A-B)\to 0
$$
we get $H^0(S, \E_x(-B))\cong\CC$.
\vskip 0.3 cm
- $\sigma$ is surjective.

Let $[E]\in M_C(2,K_C,s)$. Let us recall Mukai's Lemma 1 in 
\cite{Mukai2}. Again, we include its statement for the convenience of the reader.

\begin{lem}\label{Mukai_lem1}(Mukai) Let $E$ be a rank two vector bundle of canonical determinant $\zeta$ a line bundle on $C$. If $\zeta$ is generated by global sections, then we have
$$
\dim \Hom_{\mathcal{O}_C}(\zeta, E)\geq h^0(E)-\deg\zeta
$$
\end{lem}
Since $h^0(E)\geq s+2$, by the preceding lemma, there must be an exact sequence
$$
0\to \xi(D)\to E\to\eta(-D)\to 0
$$
for some effective divisor $D$ of degree $d$ on $C$. Since $E$ is stable we must have
$$
 \deg(\xi(D))=s+1+d\leq \deg(E)/2=2s\,,
 $$
i.e. $d \le s-1$. But then, as in the proof of Lemma \ref{dest}, we deduce that $d \le 1$. Two cases can occur. Either:
\[
0 \longrightarrow \xi(p) \longrightarrow E \longrightarrow \eta(-p) \longrightarrow 0
\]
or
\[
0 \longrightarrow \xi \longrightarrow E \longrightarrow \eta \longrightarrow 0
\]
Then one concludes exactly as in Mukai's paper \cite{Mukai2} (pp. 195-196) by using Lemma \ref{Mukai_lem1} as follows.
In the first case $E \cong E_p$ because the extension does not split and is unique. In the second case the coboundary 
$$
H^0(C, \eta)\to H^1(C, \xi)
$$ 
has rank one. We then apply point 2) in Lemma \ref{mukai2} together with the fact that, by (\ref{proj-norm-eta1}), the quadratic hull of $\Phi_{|\eta|}(C)$ is exactly $S$. We thus find a point 
$x\in S$
such that $H^0(S, I_xA)= \mathrm{Im}[H^0(S, E)\longrightarrow H^0(C, \eta)=H^0(S, A)]$. By the uniqueness again we have $E=E_x=\E_{x_{|C}}$. 
\endproof
\vskip 0.5 cm

\begin{rem}\label{stab_voisin2} \rm The two vector bundles $\wt E_L$ and $E_L$
are stable also in the Rank-2 case, for every choice of a base-point-free pencil $|L|$ of degree $s+2$.

The proof of this fact runs as follows. By Theorem \ref{caso(s+1)gon}  it is enough to prove that $\wt E_L$ is stable. Suppose not, and let 
 $N$ be a subsheaf of $\wt E_L$ with slope greater or equal than $2s=\mu_C(\wt E_L)$. Then $\alpha=N_{|C}$ destabilizes $E_L$. On the other hand, by Mukai's Lemma \ref{Mukai_lem1}, $\Hom(\xi, E_L)\neq 0$ and
we have an exact sequence 
$$
0\to\xi(D)\to E_L\to \eta(-D)\to 0
$$
for some positive divisor $D$. We may write (\ref{ineq1}) with $E$ replaced by $E_L$ and conclude, in exactly the same way, 
that $\deg(D)\leq 1$. We can then proceed as in the proof of Lemma \ref{dest} and prove that either $\alpha=\eta$ or $\alpha=\eta(-p)$. On the other hand, we have an exact sequence
$$
0\to L\to E_L\to K_CL^{-1}\to 0
$$
We must then have either $h^0(C, L\alpha^{-1})\neq0$ or $h^0(C, K_CL^{-1}\alpha^{-1})\neq0$.
For degree reasons, the only possibility is that $\alpha=\eta(-p)$ and $h^0(C, K_CL^{-1}\alpha^{-1})\neq0$.
This implies that $L=\xi(p)$ but then  $L$ can  not be base-point-free. This contradiction proves our claims.
\end{rem}

 Next, we  prepare the ground for the proof of the last step.
From the exact sequence (\ref{schwarz}) we deduce the following exact sequences

\be\label{ex_s2ab}
0\to \U\to S^2\E_x\to \mathcal{O}_S(I_x^2A^2)\to 0
\ee

\be\label{calu}
0\to \mathcal{O}_S(2B)\to \U\to \mathcal{O}_S(I_x(A+B))\to 0
\ee
In particular
$$
\U\cong\E_x(B)
$$
We also have

\be\label{ex_S2ab}
0\to U\to S^2H^0(\E_x)\to S^2H^0(I_xA)\to 0
\ee

\be\label{U}
0\to S^2H^0(S,B)\to U\to H^0(B)\otimes H^0(I_xA)\to 0
\ee

\begin{lem}\label{inj-ex-ex} $H^0(S, S^2\E_x(-C))=0$
\end{lem}

\proof We have  an exact sequence 

\be\label{s2ex-c}
0\to (S^2\E_x)(-C)\to S^2\E_x\to S^2E_x\to0
\ee
On the other hand we have an exact sequence
\be
0\to \E_x(-A)\to S^2\E_x(-C)\to \mathcal{O}_S(I_x^2(A-B))\to 0
\ee
By Lemma (\ref{a-b})  we have  $H^0(I_x^2(A-B))=0$ and we see that

\be\label{h0-exa}
H^0(\E_x(-A))=0
\ee
 by looking at the exact sequence
\be
0\to \mathcal{O}_S(B-A)\to \E_x(-A)\to I_x\to 0
\ee
\endproof

\begin{lem}\label{a-b-varie}
a) $S^2H^0(S, \mathcal{O}_S(B))\to H^0(S, \mathcal{O}_S(2B))$ is an isomorphism

b) $H^0(S, \mathcal{O}_S(B)\otimes H^0(S, I_x(A))\to  H^0(S, I_x(A+B))$ is injective

c) $F:\,S^2H^0(S, I_x(A))\to H^0(S, I^2_x(2A)$ is surjective
\end{lem}

\begin{proof}
a) Follows from the base-point-free-pencil trick.

b) Follows again from  the base-point-free-pencil trick.

c) Let  $x\in S$.
Let $A$ be a generic hyperplane section of $S$ given by the equation $s_A=0$ and assume  $x\notin A$.
Consider the following commutative diagram
$$
\xymatrix{
&&&0\ar[d]\\
&0\ar[d]&&H^0(S, I_x^2(A))\cong\CC^{s-2}\ar[d]^{\cdot s_A}\\
0\ar[r]&\ker  (F)\ar[r]\ar[d]^h&S^2H^0(S,I_xA)\ar[d]^\cong\ar[r]^F&H^0(S, I_x^2A^2)\ar[d]^r\cong\CC^{4s-5}\\
0\ar[r]&\ker  (f)\ar[r]\ar[d]&S^2H^0(A,\omega_A)\ar[r]^f&H^0(A, \omega_A^2)\cong\CC^{3s-3}\ar[d]\ar[r]&0\\
&\operatorname{coker}(h)\ar[d]&&0\\
&0
}
$$
Consider an element   $s_A\cdot t$ with $t\in H^0(S, I_x^2A)$. We can choose coordinates $\{x_0,\dots, x_s\}$ 
so that$$
x=[1,0,\dots,0]\,,\qquad s_A=x_0\,,\qquad t=x_1\,,\qquad T_x(S)=\{x_1=\cdots=x_{s-2}=0\}
$$
To prove that  $s_A\cdot t$ lies in $\operatorname{Im}(F)$ one must find $\wt Q\in S^2H^0(S,I_xA)$, i.e.  a  quadric  which is singular in $x$, such that  $\wt Q_{|S}=(x_0x_1)_{|S}$.  In other words we must find a quadric $Q\in I_S(2)$  such that
$$
\wt Q= \lambda Q+\mu x_0x_1\,,\quad \mu\neq 0
$$
is  singular in $x$. We must  have
$$
0=\left(\frac{\partial \wt Q}{\partial x_j}\right)_{x}=\lambda\left(\frac{\partial  Q}{\partial x_j}\right)_{x}\,,\qquad j\neq 1\,,\qquad 0=\left(\frac{\partial \wt Q}{\partial x_1}\right)_{x}=\lambda\left(\frac{\partial  Q}{\partial x_1}\right)_{x}+\mu
$$
Since the ideal of $S$ is generated by quadrics (Remark \ref {consq-a-b}) we may 
choose $Q$ in $I_S(2)$ such that $T_x(Q)=\{x_1=0\}$, so that  
$$
Q=x_0x_1+\sum_{i\neq0, j\neq 0} b_{ij}x_ix_j
$$
We may then set $\wt Q=Q-x_0x_1$.

\end{proof}

\vskip 0.5 cm
We are now ready to prove:

 \begin{prop}\label{dim-tan}
For every $x\in S$ 
$$
\dim T_{[E_x]}(M_C(2,K_C,s))=2
$$
and the differential of $\sigma: M_v(S)\to
M_C(2,K_C,s)$ is an isomorphism at $[\E_x]$, for every $x\in S$.
\end{prop}

\proof   In view of  Proposition \ref{L:rk1dsigma} we must verify the following three conditions:
 \begin{itemize}
	\item[(i)]    $H^1(S, S^2\mathcal{E}_x)=0$.
	
	\item[(ii)] $H^0(S, S^2\mathcal{E}_x(-C))=0$.
	
	\item[(iii)]  $S^2H^0(S, \mathcal{E}_x)\longrightarrow H^0(S, S^2\mathcal{E}_x)$ is surjective.
	\end{itemize}

We start with $(iii)$.
Look at (\ref{ex_s2ab}), (\ref{ex_S2ab}) and (\ref{U}). We get a diagram

$$
\xymatrix{
0\ar[r]&U\ar[r]\ar[d]_u&S^2H^0(S,\E_x)\ar[d]^c\ar[r]^l&S^2H^0(S, I_xA)\ar[d]^F\ar[r]&0\\
0\ar[r]&H^0(S, \U)\ar[r]&H^0(S^2\E_x)\ar[r]^m&H^0(I_x^2A^2)\ar[d]\ar[r]&H^1(S, \U)\\
&&&0\\
}
$$
We will show that $u$ is an isomorphism. 
Let us first show that

\be\label{exb}
H^1(S, \U)=0\,,\qquad \text{where}\quad \U=\E_x(B)
\ee

For this we look at the sequence
$$
0\to\E_x\to \E_x(B)\to \E_x(B)_{|B}\to 0
$$
We get:  $H^1( \U) =H^1( \E_x(B))=H^1( \E_x(B)_{|B})=H^1( {\E_x}_{|B})$.
From the exact sequence
$$
0\to\E_x(-B)\to \E_x\to{ \E_x}_{|B}\to 0
$$
we get 
$$
H^1( { \E_x}_{|B})\cong H^2( \E_x(-B))
$$
However,  $H^2( \E_x(-B))=H^0( \E_x(-A))=0$ by (\ref{h0-exa}).
In conclusion: $H^1( \U)=0$.

 We now claim that $u$ is an isomorphism.
Consider the diagram
$$
\xymatrix{
0\ar[r]&S^2H^0(S, B)\ar[r]\ar[d]&U\ar[d]_u\ar[r]&H^0(S, B)\otimes H^0(S, I_xA)\ar[d]\ar[r]&0\\
0\ar[r]&H^0(S,2B)\ar[r]&H^0(S, \U)\ar[r]^-m&H^0(S, I_x(A+B))\ar[r]&H^1(S, 2B)\ar[r]& 0
}
$$

Since  $S^2H^0(B)\to H^0(2B)$ is an isomorphism and
$$
H^0(B)\otimes H^0(I_xA)\to H^0(I_x(A+B))
$$ 
is injective, the claim follows from a dimension count.

From Lemma \ref{a-b-varie}, we know that 
$F$ is surjective and so is  $Fl$ and therefore $m$ and thus $c$, proving $iii)$.

Item  $(ii)$ is Lemma \ref{inj-ex-ex}.

To prove  $(i)$ we look at the exact sequence
$$
0\to H^1(\U)=0\to H^1(S^2\E_x)\to H^1(I_x^2A^2)
$$
But now  $ H^1(I_x^2A^2)=0$ as it follows from the exact sequences
$$
0\to I_xA^2\to A^2\to A^2_{|x}\to 0\,,\qquad 0\to I_x^2A^2\to I_xA^2\to A^2_{|x}\otimes I_x/I_x^2\to 0
$$
and from the ampleness of $A$. 
 \endproof
 
The proof of Theorem \ref{caso(s+1)gon} is now complete.


\section{Brill-Noether loci in the   Rank-1 case}\label{bn_vb1}

The purpose of this section is to prove the following:

\begin{theorem}\label{Tmain}
Let $(S,C)$ be a general pair   belonging to the Rank-1 case. 
There is a unique, generically smooth, 2-dimensional irreducible component  $V_C(2,K_C,s)$ of $M_C(2,K_C,s)$, containing the Voisin bundles $E_L$, with $L\in W^1_{s+2}(C)$,
such that 
  $\sigma$ induces   an isomorphism of $M_v(S)$ onto $V_C(2,K_C,s)_{red}$. In particular   $V_C(2,K_C,s)_{red}$ is a K3 surface.
\end{theorem}

 Before going into the proof we need     some preliminaries.  In the next statement we will refer to the notations introduced in diagram (\ref{stacks}). 

\begin{lem}\label{stackcurve}
Let $(S_0,C_0) \in \mathcal{P}_g$ be a pair belonging to the Rank-2 case. Then there exists a nonsingular affine curve B and a pair $(\mathcal{S},\mathcal{C})$  
with the following properties. There is a diagram of smooth families over $B$  
\begin{equation}\label{Bfamily}
\xymatrix{
\mathcal{C}\ar[dr] \ar@{^(->}[rr]&   & \mathcal{S} \ar[dl] \\
&B}
\end{equation}
whose fibre $(\mathcal{S}(b_0),\mathcal{C}(b_0))$  over $b_0$ is $(S_0,C_0)$, and such that for all $b \ne b_0$ outside a countable subset the fibre $(\mathcal{S}(b),\mathcal{C}(b))\in \mathcal{P}_g$ is a pair belonging to the Rank-1 case.
\end{lem}

\proof   Let $\mathcal{H}_g$ be the open subset   of the Hilbert scheme of $\PP^g$ parametrizing nonsingular K3 surfaces of degree $4s$ and let 
$\mathcal{F}_g\longrightarrow \mathcal{H}_g$ be the open subset of the flag Hilbert scheme parametrizing pairs $C \subset S \subset \PP^g$ with $[S] \in \mathcal{H}_g$ and $C \in |\mathcal{O}_S(1)|$. 
Then $(S_0,C_0)$ corresponds to a point $b_0  \in \mathcal{F}_g$. Let $B\subset \mathcal{F}_g$ be a general nonsingular affine curve through $b_0$. Then the pullback to $B$ of the universal family
over $\mathcal{F}_g$ has the required properties. \endproof

\medskip

{\it Proof of Theorem \ref{Tmain}.} The family (\ref{Bfamily}) defines naturally a varying Mukai vector $\upsilon(b)$ such that $v=\upsilon(b_0)$;    as $b \in B$ varies the moduli spaces 
$M_{\upsilon(b)}(\mathcal{S}(b))$ fit into a family 
$\varphi:\mathcal{M}_\upsilon(\mathcal{S}/B)\longrightarrow B$ of projective  surfaces. Modulo shrinking $B$ if necessary, we may assume that this is a family of K3 surfaces. Similarly, the moduli spaces $M_{\mathcal{C}(b)}(2,K_{\mathcal{C}(b)})$ fit into a smooth proper family
$\mathcal{M}_{\mathcal{C}/B}(2,\omega_{\mathcal{C}/B})\longrightarrow B$ of relative dimension $3g-3$.  By the openness of (semi)stability (\cite{Huybrechts-Lehn}, Proposition 2.3.1) and the properness of $\varphi$ we may assume that the restriction morphisms 
\[
 \sigma_b':M_{\upsilon(b)}(\mathcal{S}(b)) \longrightarrow M_{\mathcal{C}(b)}(2,K_{\mathcal{C}(b)})
\]
are well defined. They define a morphism of relative moduli spaces:
\begin{equation}\label{modoverB}
\xymatrix{
\mathcal{M}_\upsilon(\mathcal{S}/B) \ar[rr]^-{\Sigma'}\ar[dr]_-\varphi&&\mathcal{M}_{\mathcal{C}/B}(2,\omega_{\mathcal{C}/B})\ar[dl]\\
&B}
\end{equation}
 Over $b_0$ we have 
\[
\Sigma'(b_0)=\sigma_0': M_{v_0}(S_0) \longrightarrow M_{C_0}(2,K_{C_0})
\]
 which is an embedding, with image 
$M_{C_0}(2,K_{C_0},s) = V_{C_0}(2,K_{C_0},s)$.  Therefore, modulo shrinking $B$ if necessary, we may assume that \emph{for all $b\in B$} we have that $\Sigma'(b)=\sigma'_b$ embeds the K3 surface
$M_{\upsilon(b)}(\mathcal{S}(b))$ into $M_{\mathcal{C}(b)}(2,K_{\mathcal{C}(b)})$, and the image is contained in $M_{\mathcal{C}(b)}(2,K_{\mathcal{C}(b)},s)$.

Modulo performing an \'etale base change we may further assume that there is a line bundle $\mathcal{L}$ on $\mathcal{C}$ such that 
$\mathcal{L}(b) \in W^1_{s+2}(\mathcal{C}(b))$ and $L_0:=\mathcal{L}(b_0)$ is a base point free $g^1_{s+2}$. 
Consider the corresponding family $\mathcal{E_L}$ of Voisin bundles on $\mathcal{S}$. The vector bundle 
$\mathcal{E_L}(b_0)$ over $S_0$ satisfies conditions (i),(ii) and (iii) of Proposition \ref{L:rk1dsigma} as all bundles in $M_v(S_0)$ do (see Section \ref{bn_vb}). 
Therefore by upper-semicontinuity we may assume that all bundles 
$\mathcal{E_L}(b)$ satisfy at least (i) and (ii) as well. Moreover, by construction, they also satisfy $h^0(\mathcal{E_L}(b))=s+2$, so that in particular 
$S^2 H^0(\mathcal{E_L}(b))$ has  constant dimension. Moreover they also satisfy $h^0(S^2 \mathcal{E_L}(b))= h^0(S^2 \mathcal{E_L}(b)_{|\mathcal{C}(b)})$-2, as shown by the exact sequence: 
\[
\xymatrix{
0 \ar[r] & S^2 \mathcal{E_L}(b)(-\mathcal{C}(b)) \ar[r] & S^2 \mathcal{E_L}(b) \ar[r] & S^2 \mathcal{E_L}(b)_{|\mathcal{C}(b)} \ar[r] & 0
}
\]
because $h^1( S^2 \mathcal{E_L}(b)(-\mathcal{C}(b)))=2$. 
Therefore semicontinuity applies and condition (iii) can be also assumed to be satisfied for all $b \in B$. 

We now apply Proposition \ref{L:rk1dsigma} and we deduce that $M_{\mathcal{C}(b)}(2,K_{\mathcal{C}(b)},s)$ is smooth of dimension 2 at $\sigma_b(\mathcal{E_L}(b))$. 
Therefore $\sigma_b$ embeds $M_{v(b)}(\mathcal{S}(b))$ into an irreducible 2-dimensional generically smooth component $V_{\mathcal{C}(b)}(2,K_{\mathcal{C}(b)},s)$
 of $M_{\mathcal{C}(b)}(2,K_{\mathcal{C}(b)},s)$ whose reduction is therefore isomorphic to 
$M_{v(b)}(\mathcal{S}(b))$.  This component is uniquely determined by the condition of containing the bundles $E_{\mathcal{L}(b)}$. \qed


\section{The Fourier-Mukai transform}\label{FM}

As usual we consider a pair $(S,C)$   which we assume to be either in the Rank-1  or in the Rank-2 case. If we are in the Rank-1 case we denote by $T$ the K3 surface  $V_C(2,K_C,s)_{red}$ introduced in the previous section. We do the same in the Rank-2 case where, by virtue of Theorem \ref{caso(s+1)gon}, we have 
$V_C(2,K_C,s)_{red}=M_C(2,K_C,s)$. In both cases we have an isomorphism
$$
\sigma:  M_v(S) \overset\cong\longrightarrow T\,,\qquad\text{where }\quad v=(2,[C], s)
$$
We will always view the K3 surface $T$ as a sub variety of $M_C(2, K_C)$.
We further  assume that
$$
s=2t+1\,,\qquad\text{i.e.}\quad g\equiv 3, \mod 4
$$
Following Mukai's program (Remark 10.3 in \cite{Mukai-BN-Fano}) and its implementation in genus eleven (Section 4 of \cite{Mukai2}) we are going to prove the following theorem.

\begin{theorem}\label{torelli_type}
There exists
a Poincar\'e bundle
$$
\xymatrix{\U\ar[d]\\
C\times T
}
$$
{\it unique up to isomorphism},  having the following properties.
Denote by  $\pi_C: C\times T\to C$ and  $\pi_T: C\times T\to T$ the two  projections, then

\begin{itemize}
	\item[i)]   $ \U_{|C\times\{[E]\}}\cong E\,,\quad \forall \,\,\, [E]\in T\subset M_C(2, K_C)$
	
	\item[ii)]  $\det(\U)\cong K_C\boxtimes h_{det}$, where $ h_{det}=(\det R^1{\pi_T}_*\U)\otimes(\det{\pi_T}_*\U)^{-1}$

	\end{itemize}
	
	Moreover:
	
	\begin{itemize}
	
	\item[iii)]  $h_{det}$ is a polarization of genus $g$ on $T$.
	
	\item[iv)]  For each $x\in C$, the vector bundle $\U_x=\U_{|\{x\}\times T}$ is stable and $[\U_x]\in M_{\wh v}(T)$, where 
	$\wh v=(2,h_{det},s)$
	
	  \item[v)] The morphism $C\longrightarrow \wh T= M_{\wh v}(T)$ defined by $x\mapsto \U_x$
	  is an embedding.
	  \item[vi)] The  Fourier-Mukai transform  $(\wh T, \wh h_{det})$ of $(T, h_{det})$ is isomorphic to $(S, h)$, where $h=[C]$.
		 
\end{itemize}
\end{theorem}

\proof
Write
$$
M_C(2, K_C)=R/P GL(\nu)\,,\qquad R\subset \operatorname{Quot}
$$
consider the quotient  map 
$$
p: R\to M_C(2, K_C)
$$
and set
$$
R'=p^{-1}(T)
$$
Let $\wt\U$ be the restriction to $C\times R'$ of the universal bundle over $C\times  \operatorname{Quot}$. 
Consider the projection
$$
\pi_{R'}: C\times R'\to R'
$$

The sheaf
${\pi_{R'}}_*\wt\U$ is a vector bundle of rank $s+2$, while  ${\pi_{R'}}_*(\wt\U\boxtimes K_C) $ is a vector bundle of rank $8s$:
indeed $h^0(E\otimes K_C)=3\deg K_C+2(1-g)=8s$. Since $s=2t+1$, the two integers $8s$ and $s+2$ are relatively prime
and we can find integers $x$ and $y$ such that $1+x(s+2)+y8s=0$.

Consider then the vector bundle on $C\times R'$:
$$
\V= \wt\U\boxtimes\pi_{R'}^*(\det({\pi_{R'}}_*\wt\U)^{x}\otimes \det( {\pi_{R'}}_*(\wt\U\boxtimes K_C))^{y})
$$
The action of a central element $c\in \CC^*\subset GL(\nu)$ on the three factors are :  $1$, $c^{x(s+2)}$ and $c^{8ys}$.
Thus the vector bundle   $\V$ is acted on by  $\PP GL(\nu)$ and descends to a Poincar\'e bundle $\V$ on $C\times T$.
Since $T$ is regular there exists a line bundle $L$ on $T$ such that
$$
\det \V=K_C\boxtimes L
$$
Thus
$$
\V^\vee\boxtimes \pi_C^*K_C\cong \V\boxtimes \pi_T^*L^{-1}
$$
As a consequence by Serre duality: 
$$
(R^1{\pi_{T}}_*\V)^\vee\cong{\pi_T}_*(\V^\vee\boxtimes \pi_C^*K_C)\cong{\pi_T}_*\V\otimes L^{-1}
$$
Hence
$$
h_{det}=(\det R^1{\pi_T}_*\V)\otimes(\det{\pi_T}_*\V)^{-1}\cong L^{s+2}\otimes (\det{\pi_T}_*\V)^{-2}
$$
Now the universal bundle 
$$
\U=\V\boxtimes L^{t+1}\boxtimes  (\det{\pi_T}_*\V)^{-1}
$$
satisfies both $i)$  and  $ii)$. The unicity follows from the fact that $\Pic (T)$ is torsion free.

Since properties $iii)$,  $iv)$ and  $v)$ are invariant under small deformations, we may limit ourselves to the rank-two case.
In this case we have the universal extension
\be\label{univ-ext}
0\to \mathcal{O}_{S\times T}(\rho^*B)\to \F\to \I_\Delta(\rho^*A+\tau^*(B-A)) \to 0
\ee
Where  $\rho$ and $\tau$ are the projections $S\times T\to S$ and $S\times T\to T$. Moreover we identify $S$ and $T$ via the 
isomorphism 
\be\label{ident-s-t}
\aligned
S&\longrightarrow M_v(S)\overset \sigma\longrightarrow T\\
\qquad&x\mapsto\,\,\, \E_x\,\,\,\mapsto\,\,\,  E_x={\E_x}_{|C}
\endaligned
\ee
(c.f. Theorem \ref{caso(s+1)gon} and Proposition  \ref{S-to-M}). 
We have
$$
\det(\F_{|C\times T})=K_C\boxtimes\mathcal{O}(B-A)
$$
Consider
$$
0\to \tau_*\mathcal{O}_{S\times T}(\rho^*B)\to\tau_* \F\to \tau_* \I_\Delta(\rho^*A+\tau^*(B-A))\to R^1\tau_*\mathcal{O}_{S\times T}(\rho^*B)\to0
$$
which gives
$$
0\to H^0(\xi)\otimes \mathcal{O}_T\to\tau_*(\F_{C\times T})\to H^0(\eta)\otimes \mathcal{O}_T(B-A)\to\mathcal{O}_T(B)\to 0
$$
We know that
$$
h_{\det}=L^{s+2}\otimes\det \tau_*(\F_{C\times T})^{-2}
$$
We then have
$$
L=\mathcal{O}(B-A)\,,\qquad  \det \tau_*(\F_{C\times T})=\mathcal{O}(sB-(s+1)A)\,,\qquad h_{\det}=\mathcal{O}(sA-(s-2)B)
$$
Thus $h_{\det}$ is a positive polarization and its genus is given by
$$
g(h_{\det})=\frac{1}{2}(sA-(s-2)B)^2+1=2s+1\,.
$$
proving $iii)$.

In the rank-two case the normalized Poincar\'e bundle is given by 
$$
\U=\F_{|C\times T}\boxtimes \tau^*\mathcal{O}((t+1)A-tB)
$$
Let $C'$ be a smooth element in $|sA-(s-2)B|$. Set $B'=(t+1)A-tB$ and $A'=tA-(t-1)B$.
Under the identification given by (\ref{ident-s-t}) we consider $C'$, $A'$ and $B'$ as divisors in $T$.
We may then consider the rank-two case given by the decomposition
$$
\Pic(T)=\ZZ\cdot A'\oplus\ZZ\cdot B'\,,\qquad |C'|=|A'+B'|
$$
For this case the universal extension can be given by tensoring (\ref{univ-ext}) by $ \tau^*\mathcal{O}(B')$.
For each $x\in S$ setting $\U_x=\U_{|\{x\}\times T}$, we get
$$
0\to\mathcal{O}_T(B')\to\U_x\to I_x(A')\to 0
$$
We also get an isomorphism  (Theorem \ref{caso(s+1)gon})
\be\label{ident-s-t2}
\aligned
T&\longrightarrow M_{\wh v}(T)\\
\qquad&x\mapsto\,\,\, \U_x
\endaligned
\ee
and a fortiori an embedding
\be\label{emb-c-t2}
\aligned
C&\longrightarrow M_{\wh v}(T)\\
\qquad&x\mapsto\,\,\, \U_x
\endaligned
\ee
Finally we want to show that  
$(\wh T, \wh h_{det})=(M_{\wh v}(T), \wh h_{det})$ may be identified with  $(S,h)$.
Since we  have the  isomorphism $\sigma: \wh S= M_v(S)\to T$, we have $(\wh T, \wh h_{det})=(\wh{\wh S},h')=(S,h')$
for some polarization $h'$ of genus $g$. In the rank one case we necessarily have $h'=h$. Let us show that also in the rank two case we have $\wh h_{det}=h$. To simplify notation we will prove the equivalent statement that 
$\wh h=h_{det}$, which means
\be
\wh {[C]}=[C']=[sA-(s-2)B]
\ee
\vskip 0.3 cm
From \cite{Mukai-Duality}, we recall the procedure one has to follow to construct $\wh h$, starting from $h$.

We let $\wh S=M_v(S)$, where 
$$
v=(2, h, s)
$$
Then $\wh S$ is again a K3 surface and there is a universal family $\F$ on $\wh S\times S$.
Let 
$$
c_1(\F)=h+\phi\in H^2(S)\oplus H^2(\wh S)\,,\qquad \text{and} \quad c_2^{mid}(\F)
\in H^2(S)\oplus H^2(\wh S)
$$
be the first Chern class of $\E$ and the middle K\"unneth component of the second Chern class respectively.
Define a class $\psi\in H^2(\wh S)$ by
$$
h\cup  c_2^{mid}(\F)=p\otimes \psi\in H^4(S)\oplus H^2(\wh S)
$$
where $p$ is the fundamental class of $S$. Both $\phi$ and $\psi$ are algebraic by Lefschetz theorem.
Then the class $\wh h$ is given by:
$$
 \wh h=\psi-2s\phi
$$
We now consider the rank two case in which  $\Pic(S)\cong \ZZ\cdot A\oplus\ZZ\cdot B$ and we look at
the exact sequence (\ref{univ-ext})

\be
0\to\mathcal{O}_{S\times \wh S}(p^*B)\to \F\to I_\Delta(p^*A+q^*(B-A))\to 0
\ee
We get
$$
c_1(\F)=c_1(p^*(A+B))+c_1(q^*(B-A))
$$
Thus
$$
h=c_1(p^*(A+B))\,,\qquad \phi=c_1(q^*(B-A))
$$

On the other hand
$$
c_2(\F)=c_1(p^*(B))\cup (c_1(p^*(A)+c_1(q^*(B-A))+\Delta
$$
so that
$$
c_2(\F)^{mid}=c_1(p^*(B))\cup (c_1(q^*(B-A))+\Delta
$$
Therefore, as a class in $H^2(\wh S)$,
$$
\psi=((A+B\cdot B)[B-A]+[A+B]=(s+2)B-sA
$$
As a conclusion
$$
\wh h=sA-(s-2)B.
$$

\endproof

\bibliographystyle{abbrv}

\bibliography{bibliografia}

\begin{thebibliography}{10}

\bibitem{Aprodu-Farkas-Ortega}
M.~Aprodu, G.~Farkas, and A.~Ortega.
\newblock {Minimal resolutions, Chow forms of K3 surfaces and Ulrich bundles}.
\newblock {Available at http://arxiv.org/pdf/1212.6248.pdf}, {2013}.

\bibitem{Arbarello-Cornalba}
E.~Arbarello and M.~Cornalba.
\newblock Footnotes to a paper of {B}eniamino {S}egre: ``{O}n the moduli of
  polygonal curves and on a complement to the {R}iemann existence theorem''
  ({I}talian) [{M}ath. {A}nn. {\bf 100} (1928), 537--551;\ {J}buch {\bf 54},
  685].
\newblock {\em Math. Ann.}, 256(3):341--362, 1981.
\newblock The number of $g^{1}_{d}$'s on a general $d$-gonal curve, and the
  unirationality of the Hurwitz spaces of $4$-gonal and $5$-gonal curves.

\bibitem{Ballico-Fontanari-Tasin}
E.~Ballico, C.~Fontanari, and L.~Tasin.
\newblock Singular curves on {$K3$} surfaces.
\newblock {\em Sarajevo J. Math.}, 6(19)(2):165--168, 2010.

\bibitem{Cil-Lop-Mir}
C.~Ciliberto, A.~Lopez, and R.~Miranda.
\newblock Projective degenerations of {$K3$} surfaces, {G}aussian maps, and
  {F}ano threefolds.
\newblock {\em Invent. Math.}, 114(3):641--667, 1993.

\bibitem{Ciliberto-Pareschi}
C.~Ciliberto and G.~Pareschi.
\newblock Pencils of minimal degree on curves on a {$K3$} surface.
\newblock {\em J. Reine Angew. Math.}, 460:15--36, 1995.

\bibitem{Green-Lazarsfeld-special}
M.~Green and R.~Lazarsfeld.
\newblock Special divisors on curves on a {$K3$} surface.
\newblock {\em Invent. Math.}, 89(2):357--370, 1987.

\bibitem{Huybrechts-Lehn}
{Huybrechts, D. and Lehn, M.}
\newblock {\em {The geometry of moduli spaces of sheaves}}.
\newblock Number {E $31$} in {Aspects of Mathematics}. {Vieweg}, {1997}.

\bibitem{Knutsen-smooth}
A.~L. Knutsen.
\newblock Smooth curves on projective {$K3$} surfaces.
\newblock {\em Math. Scand.}, 90(2):215--231, 2002.

\bibitem{Lelli-Chiesa}
M.~Lelli{-}Chiesa.
\newblock {Stability of rank-3 Lazarsfeld-Mukai bundles on K3 surfaces}.
\newblock {Available at http://arxiv.org/pdf/1112.2938.pdf}, {2013}.

\bibitem{Mukai2}
S.~Mukai.
\newblock Curves and {$K3$} surfaces of genus eleven.
\newblock In {\em Moduli of vector bundles ({S}anda, 1994; {K}yoto, 1994)},
  volume 179 of {\em Lecture Notes in Pure and Appl. Math.}, pages 189--197.
  Dekker, New York, 1996.

\bibitem{Mukai-Duality}
S.~Mukai.
\newblock Duality of polarized {$K3$} surfaces.
\newblock In {\em New trends in algebraic geometry ({W}arwick, 1996)}, volume
  264 of {\em London Math. Soc. Lecture Note Ser.}, pages 311--326. Cambridge
  Univ. Press, Cambridge, 1999.

\bibitem{Mukai-BN-Fano}
S.~Mukai.
\newblock Non-abelian {B}rill-{N}oether theory and {F}ano 3-folds [translation
  of {S}\=ugaku {\bf 49} (1997), no. 1, 1--24; {MR}1478148 (99b:14012)].
\newblock {\em Sugaku Expositions}, 14(2):125--153, 2001.
\newblock Sugaku Expositions.

\bibitem{Mukai}
{Mukai, S.}
\newblock {Symplectic structure of the moduli space of sheaves on an abelian or
  $K3$ surface}.
\newblock {\em {Invent. Math.}}, {77}({1}):101--116, {1984}.

\bibitem{Saint-Donat}
{Saint-Donat, B.}
\newblock {Projective models of {$K3$} surfaces}.
\newblock {\em {Amer. J. Math.}}, {96}:602--639, {1974}.

\bibitem{Voisin_W}
C.~Voisin.
\newblock Sur l'application de {W}ahl des courbes satisfaisant la condition de
  {B}rill-{N}oether-{P}etri.
\newblock {\em Acta Math.}, 168(3-4):249--272, 1992.

\bibitem{Wahl-square}
J.~Wahl.
\newblock On cohomology of the square of an ideal sheaf.
\newblock {\em J. Algebraic Geom.}, 6(3):481--511, 1997.

\bibitem{Wahl-sections}
J.~Wahl.
\newblock Hyperplane sections of {C}alabi-{Y}au varieties.
\newblock {\em J. Reine Angew. Math.}, 544:39--59, 2002.

\end{thebibliography}

\end{document}